\documentclass[letterpaper, 10 pt, conference]{ieeeconf}  % Comment this line out if you need a4paper

\IEEEoverridecommandlockouts                              % This command is only needed if 
\usepackage{amsmath} % assumes amsmath package installed

\usepackage{amssymb,amsfonts,amsthm,dsfont} % math
\usepackage{graphics}
\usepackage[dvips]{graphicx}
\usepackage[table,usenames,dvipsnames]{xcolor}
\usepackage{cancel}
\usepackage{subcaption}
\usepackage{tikz-cd}
\usepackage{algorithm2e}
\usepackage{epstopdf}
\usepackage{tikz}
\usetikzlibrary{decorations.pathmorphing} % for snake lines
\usetikzlibrary{matrix} % for block alignment
\usetikzlibrary{arrows} % for arrow heads
\usetikzlibrary{calc} % for manimulation of coordinates

\tikzstyle{block} = [draw,rectangle,thick,minimum height=2em,minimum width=2em]
\tikzstyle{connector} = [->,thick]
\tikzstyle{line} = [thick]
\tikzstyle{branch} = [circle,inner sep=0pt,minimum size=1mm,fill=black,draw=black]

{
      \theoremstyle{plain}
      \newtheorem{assumption}{Assumption}
  }

{
      \theoremstyle{plain}
      
  }
\newtheorem{theorem}{Theorem}
\newtheorem{lemma}{Lemma}

\theoremstyle{definition}

\newcommand{\ba}{\begin{align}}
\newcommand{\ea}{\end{align}}
\newcommand{\fr}{\frac}

\newcommand{\lam}{\lambda}
\newcommand{\R}{{\mathbb R}}

\title{\LARGE \bf Neuron Growth Output-Feedback Control by PDE Backstepping
}

\author{Cenk Demir$^{1}$ Shumon Koga$^{2}$ and Miroslav Krstic$^{1}$% <-this % stops a space
\thanks{$^{1}$Cenk Demir and Miroslav Krstic are with the
Department of Mechanical and Aerospace Engineering, U.C. San Diego,
9500 Gilman Drive, La Jolla, CA, 92093-0411, 
        {\tt\small cdemir@ucsd.edu} and
{\tt\small krstic@ucsd.edu}}%
\thanks{$^{2}$Shumon Koga is with the Department of Electrical and Computer Engineering, U.C. San Diego,
9500 Gilman Drive, La Jolla, CA, 92093-0411,
        {\tt\small skoga@ucsd.edu}}%
}

\begin{document}

\maketitle
\thispagestyle{empty}
\pagestyle{empty}

%%%%%%%%%%%%%%%%%%%%%%%%%%%%%%%%%%%%%%%%%%%%%%%%%%%%%%%%%%%%%%%%%%%%%%%%%%%%%%%%
\begin{abstract}

Neurological injuries predominantly result in loss of functioning of neurons. These neurons may regain function after particular medical therapeutics, such as Chondroitinase ABC (ChABC), that promote axon elongation by manipulating the extracellular matrix, the network of extracellular macromolecules, and minerals that control the tubulin protein concentration, which is fundamental to axon elongation. We introduce an observer for the concentration of unmeasured tubulin along the axon, as well as in the growth cone, using the measurement of the axon length and the tubulin flux at the growth cone. We employ this observer in a boundary control law which actuates the tubulin concentration at the soma (nucleus), i.e., at the end of the axon distal from the measurement location. For this PDE system with a moving boundary, coupled with a two-state ODE system, we establish global exponential convergence of the observer and local exponential stabilization of the [axon, observer] system in the spatial $\mathcal{H}_1$-norm. The results require that the axon growth speed be bounded. For an open-loop observer, this is ensured by assumption (which requires that tubulin influx at the soma be limited), whereas for the output-feedback system the growth rate of the axon is ensured by assuming that the initial conditions of all the states, including the axon length, be sufficiently close to their setpoint values. 

%We introduce  observer  and output-feedback control for axon elongation by regulating the tubulin concentration at the neuron's soma (nucleus). Tubulin dynamics in the axon and the axon length are modeled as coupled parabolic diffusion-reaction-advection Partial Differential Equations (PDE) with a boundary governed by Ordinary Differential Equations (ODE). A novel backstepping method for observer design and output-feedback control for this coupled PDE-ODE with moving boundary problem is proposed under the measured outputs of the axon's length and tubulin concentration change. After transforming the observer error system to a target error system, the method of successive approximation is applied to ensure the well-posedness of the solution to the conditions for gain kernel function, which can be solved numerically. By applying Lyapunov analysis to the target system in the spatial $\mathcal{H}_1$-norm, we prove that the observer error system is globally exponentially stable, and the closed-loop system with the proposed output-feedback controlled law is locally exponentially stable.

\end{abstract}

%%%%%%%%%%%%%%%%%%%%%%%%%%%%%%%%%%%%%%%%%%%%%%%%%%%%%%%%%%%%%%%%%%%%%%%%%%%%%%%%
\section{INTRODUCTION}
Neuroscience is a complex interdisciplinary research field  aimed at enhancing understanding of how neurons perceive information and curing nervous system-related disorders and damage \cite{izhikevich2007dynamical,  squire2012fundamental,ribar2020neuromorphic}. These disorders and impairments, such as Parkinson's disease, Alzheimer's disease, Huntington's disease, and spinal cord injuries, may occur because of the degeneration of neurons  \cite{maccioni2001molecular,dauer2003parkinson,liu1997neuronal}. These medical disorders regularly result from the loss of functionality of neurons, such as the cessation of elongation. Treatments such as ChABC can heal neurons and resume their activities\cite{karimi2010synergistic, bradbury2011manipulating}. The basic idea of this therapy is to manipulate the extracellular matrix (ECM) to regulate the activity of neurons as desired \cite{frantz2010extracellular}. 

Neurons are highly active cells whose mission is to receive and transmit electrical signals which contain perceptual information. This process begins with a signal entry from the dendrites of a postsynaptic neuron. It continues with the transportation of the signal through the axon by using the energy produced in the soma. During transmission, the growth cone, located at the end of the axon, seeks the chemical cues to detect the neuron that will receive the signal \cite{julien1999neurofilament}. After detecting the direction, tubulin dimers and monomers assemble to create microtubules which extend the axon with the assistance of ECM \cite{ barros2011extracellular}. Thus, tubulin concentration dynamics in the axon are the primary regulator of axon elongation. These dynamics express the following behaviors: the tubulin production in the soma, assembly and disassembly processes of microtubules in the axon, and the transportation process along the axon and in the growth cone \cite{DIEHL2014194}. Recent preclinical studies show that manipulation of ECM controls the axon growth, by which the tubulin concentration is controlled \cite{burnside2014manipulating}.
 
Researchers proposed different mathematical models to explain the axon growth process by considering most or some of the behavior of tubulin. One of the pioneer mathematical models describes the polymerization of tubulin  \cite{buxbaum1988thermodynamic}. In another model developed in \cite{van1994neuritic}, the authors propose tubulin transportation as diffusion and include the axon growth due to the polymerization process. A PDE model of axon growth is presented in \cite{mclean2004continuum}, and its stability properties are defined in \cite{mclean2006stability}. Another pioneer axon growth model introduces a coupled PDE-ODE with a moving boundary and gives numerical results for tubulin concentration along the axon \cite{diehl2014one}.

Besides computational purposes, the coupled PDE-ODE axon growth model has also begun to be studied from a control-theoretical perspective to stabilize the axon growth \cite{demir2021neuron}. Recent enhancements related to the boundary control of PDEs have motivated researchers in different fields \cite{krstic2008boundary}. The contribution of \cite{smyshlyaev2004closed} has introduced the method of successive approximation to obtain solutions to kernel PDEs for Volterra type of transformation for parabolic PDEs. These groundbreaking studies have extended the type of systems accompanied with boundary control to the class of coupled PDE-ODE systems \cite{susto2010control, tang2011state}. While almost all of the studies dealt with a constant domain size in time, the authors of \cite{koga2018control, krstic2020materials} have developed a backstepping design for a parabolic PDE with a moving boundary, called the Stefan problem. The authors have proven the global exponential stability of the closed-loop system with the proposed input by means of both  state-feedback and output-feedback control. In addition, recent research considers nonlinear PDE systems and proves the stability results in a local sense \cite{coron2013local,buisson2018control}.  While those results about local stabilization for nonlinear PDE systems have been achieved for hyperbolic PDEs, an output-feedback stabilization for the coupled nonlinear parabolic PDE-ODE with a moving boundary, which is a class of the system modeling axon growth dynamics proposed in literature, has not been studied so far. 

This paper present the output-feedback stabilization for the tubulin concentration model associated with the axon growth dynamics, described by a parabolic PDE with moving boundary governed by a nonlinear ODE. First, we obtain the linearized reference error system by considering the error variable from the steady-state solution of the system in the plant dynamics for a given desired axon length, and applying the linearization to ODE state in order to deal with algebraic nonlinearity. By setting the measured output of the system as the axon length and the change of the tubulin concentration in growth cone, we design an observer to estimate the plant state with an observer gain which is given by the kernel functions via the backstepping technique. Since these kernels are not analytically solvable, we apply the method of successive approximation to guarantee the well-posedness of the solution. We rigorously prove the global stability for the observer error system, and the local stability result for the closed-loop system with the output-feedback control following the similar procedure to \cite{demir2021neuron}. The numerical simulation is performed to investigate the performance of the proposed observer and output-feedback control designs, which illustrates the desired performance in both regulation of the axon growth and the estimation of the tubulin concentration. 

The paper is structured as follows. Section \ref{sec:section2} introduces the coupled PDE-ODE system with a moving boundary modeling the tubulin concentration and the axon growth. Section \ref{sec:observer} presents the observer design via backstepping method, the method of successive approximation, and the stability analysis of the estimation error system. Section \ref{sec:section4} proposes the observer-based output-feedback control and the stability proof of closed-loop system. Section \ref{sec:section5} provides the simulation result for the closed-loop system of the axon growth model using the verified physical parameters with the proposed observer and output-feedback control. The paper ends with the conclusion in Section \ref{sec:section6}.

\begin{figure}[t]
\centering
       \includegraphics[width=0.45\linewidth]{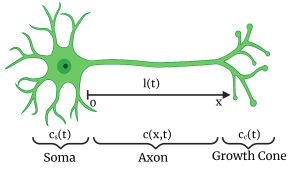}
  \caption{Schematic of neuron and state variables}
  \label{fig:1b} 
\end{figure}

\section{MODELING OF AXON GROWTH}
\label{sec:section2}

This section presents the mathematical model of the axon elongation process governed by a coupled PDE-ODE with a moving boundary. In this model, the concentration of tubulin protein along the axon directly controls the growth of a newborn axon. There are two underlying assumptions for modeling the axonal growth given here. The first assumption is that the tubulin protein is entirely responsible for the growth of an axon, not any other substance. The second assumption is that free tubulin molecules are modeled as homogeneous continuums because of their small size. 

Let $l(t)$ denote the length of axon, $x$ denote the one-dimensional coordinate along the axon, and $c(x,t)$ denote the tubulin concentration along this one-dimensional coordinate. The variables with subscripts $c$, and $s$ describe the ones of growth cone and soma, namely, $c_{\rm c}(t)$ and $c_{\rm s}(t)$ are the tubulin concentration in the growth cone and soma, respectively, as shown in Fig. \ref{fig:1b}. While free tubulin proteins move with the constant velocity $a$, and diffuse with the diffusivity constant $D$, the degradation along the axon occurs with the constant rate $g$. The constant $l_{\rm c}$ is the growth ratio of growth cone, $\tilde{r}_{\rm g}$ is the chemical reaction rate of free tubulin monomers and dimers to create microtubules, which is shown in Fig. \ref{fig:1a}. Taking into account all of these effects, the dynamics of the tubulin concentration associated with the dynamics of the axon length are described as
\begin{align}\label{sys1} 
c_t (x,t) =& D c_{xx} (x,t) - a c_x (x,t) - g c(x,t) , \\
\label{sys2} c_x(0,t) = & - q_{\rm s}(t), \\
\label{sys3} c(l(t),t) =& c_{\rm c} (t), \\
\label{sys4} l_{\rm c} \dot{c}_{\rm c}(t) = & (a-gl_{\rm c}) c_{\rm c}(t) - D c_x (l(t), t) \notag\\
& - (r_{\rm g} c_{\rm c}(t) + \tilde{r}_{\rm g} l_{\rm c} )(c_{\rm c}(t) - c_{\infty}), \\
\label{sys5} \dot{l}(t) =& r_{\rm g} (c_c(t)-c_{\infty}), 
\end{align}
where $r_{g}$ is a lumped parameter introduced in \cite{DIEHL2014194}, and $c_{\infty}$ is a equilibrium concentration in the cone to stop the axonal growth. 

\begin{figure}[t]
\centering
\includegraphics[width=0.67\linewidth]{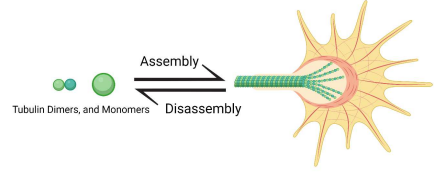}
       \caption{Tubulin assembly-disassembly}
       \label{fig:1a}
\end{figure}

Let $c_{\rm eq}(x)$ be the tubulin concentration profile for a given constant axon length $l_{\rm s}$, the solution of which is given in \cite{diehl2014one}. 
%In addition, the substitution of steady-state solution into \eqref{sys1}-\eqref{sys5} and linearization procedure around zero states are detailed in \cite{krstic2020materials}.
Let $u(x,t)$, $z_1(t)$, $z_2(t)$, and $U(t)$ be the reference error states and input defined below: 
\begin{align}
\label{eqn:sys-trans1}
u(x,t) =& c(x,t) - c_{\rm eq}(x), \\
z_{1}(t) =& c_{\rm c}(t) - c_{\infty}, \\
z_2(t) =& l(t) - l_{\rm s}, \\
U(t) = & - ( q_{\rm s}(t) - q_{\rm s}^*). 
\label{eqn:sys-trans4}
\end{align}
By subtracting the steady-state solution from  \eqref{sys1}-\eqref{sys5}, and using \eqref{eqn:sys-trans1}-\eqref{eqn:sys-trans4}, we obtain the reference error system as in \cite{demir2021neuron}. The reference error system has algebraic nonlinearity in ODE of the dynamics, so the linearization technique mentioned above is applied to deal with the nonlinearity. Then, we obtain the following dynamics (see Section 12-2 in \cite{krstic2020materials} for the details):
 \begin{align} \label{ulin-PDE}
u_t (x,t) =& D u_{xx} (x,t) - a u_x (x,t) - g u(x,t) , \\
u_x(0,t) = & U(t), \label{ulin-BC1} \\
\label{linreferr3}u(l(t),t) =&H^T X(t)  , \\
\dot{X}(t) = & A X(t) + Bu_x (l(t), t), \label{ulin-ODE}
 \end{align}
where $ X(t)=[ z_1(t) \quad z_2(t)]^\top$ and
\begin{align} \label{AB-def} 
 A = &\left[ 
 \begin{array}{cc}
 \tilde a & 0 \\
 r_{\rm g} & 0
 \end{array}  
 \right] , \quad B =  \left[ 
 \begin{array}{c}
 - \beta \\
 0
 \end{array}  
 \right], \\
H = &\left[1 \quad - \frac{(a-gl_{\rm c}) c_{\infty}}{D}\right]^T  .  \label{C-def} 
\end{align}

\textbf{Control Design Task:} Develop an observer-based output-feedback control law for the input $q_{\rm s}(t)$ so that $l(t)$ converges to a desired (setpoint) axon length $l_{\rm s}>0$, at least starting from $c(x,0)$ sufficiently near $c_{\rm eq} (x)$ (in a suitable norm in $x$) and $l(0)$ sufficiently near $l_{\rm s}$. 
%, subject to the governing equations \eqref{sys1}--\eqref{sys5}. 

\section{STATE ESTIMATION DESIGN}
\label{sec:observer}
In this section, first we state the assumptions on the axon length. Then, we propose the first major theorem in this paper with providing its proof, the exponential stability of the observer error system. 
\begin{assumption}
The axon length $l(t)$ maintains positive and is upper bounded, i.e., there exists a positive constant $\bar l>0$ such that the following inequality holds: 
\begin{align} 
    \label{ineq-l} 
     0 < l(t) \leq \bar l, 
\end{align}
for all $t \geq 0$. 
\label{asm:assump1}
\end{assumption}
\begin{assumption}
The time derivative of the axon length is also bounded, i.e., there exists a positive constant $\bar v>0$ such that the following inequality holds: 
\begin{align} 
    \big|\dot l(t) \big| \leq \bar v, \label{ineq-ldot}  
\end{align}
for all $t\geq 0$.
\label{asm:assump2}
\end{assumption}
Before we state our major theorem, we define the $\mathcal{H}_1$-norm as $||f(.,t)||_{H_1}=\sqrt{\left(\int_0^{l(t)}f^2(.,t)+f_x^2(.,t)dx\right)}$.
\begin{theorem}
Let Assumptions \ref{asm:assump1} and \ref{asm:assump2} hold. Consider the plant \eqref{ulin-PDE}-\eqref{ulin-ODE} and the available measurements
%and the quantities assumed to be measured and treated as the system's outputs are 
\begin{align}
    y_1(t)=u_x(l(t),t), \quad 
    y_2(t)=C X(t) \label{measurements}
\end{align}
where %$C \in \mathbb{R}^2$, and
$C=\left[0 \quad 1\right]$. Also, consider the observer designed as
%Then, the observer is designed as
 \begin{align} 
 \label{obulin-PDE}
\hat{u}_t (x,t) =& D \hat{u}_{xx} (x,t) - a \hat{u}_x (x,t) - g \hat{u}(x,t)\nonumber \\ &+p_1(x,l(t))\left(u_x(l(t),t)-\hat{u}_x(l(t),t)\right) , \\
\hat{u}_x(0,t) = & U(t), \label{obulin-BC1} \\
\label{oblinreferr3}\hat{u}(l(t),t) =&H^T \hat{X}(t)
, \\
\dot{\hat{X}}(t) = & A \hat{X}(t) + B u_x (l(t), t)+LC(X(t)-\hat{X}(t)) 
\label{obulin-ODE} ,
 \end{align}
where $x\in[0,l(t)]$, $L$ is chosen to make $A-LC$ Hurwitz, and the observer gain $p_1(x,l(t))=DP(x,l(t))$ where $P(x,l(t))$ is the solution to the following PDE
\begin{align}
\label{KernelPDE1}
 DP_{yy}(x,y)-DP_{xx}(x,y)+aP_x&(x,y)\nonumber \\
 -aP_y(x,y)&=\lambda P(x,y),   \\
 P(x,x)&=\left(\frac{\lambda}{2D}x+\gamma_1\right),  \label{kernelPDE2}\\
 P_x(0,y)&=0 , 
 \label{kernelPDE3}
\end{align}
where $\lambda>0$ is an arbitrary constant, and $\gamma_1 $ is a constant satisfying $\frac{D}{a}\leq\gamma_1$.
Then, the observer error system is exponentially stable in $\mathcal{H}_1$-norm, i.e., there exist positive constants $M>0$ and $\kappa>0$ such that the following norm estimate holds: 
\begin{align}
    \tilde \Phi(t) \leq M \tilde \Phi(0) e^{- \kappa t}, 
\end{align}
where $\tilde \Phi(t):= || u - \hat u||_{H_1} + |X - \hat X|$. 
\label{thm:teorem1}
\end{theorem}

Theorem \ref{thm:teorem1} is proved in the remainder of this section.

\subsection{Observer design and backstepping transformation}
\subsubsection{Observer design and observer error system}

Then, we define the observer error state as 
\begin{align}
    \tilde{u}(x,t)=u(x,t)-\hat{u}(x,t), \quad
    \tilde{X}(t)=X(t)-\hat{X}(t).
\end{align}
Subtracting the observer system \eqref{obulin-PDE}-\eqref{obulin-ODE} from the plant \eqref{ulin-PDE}-\eqref{ulin-ODE}, the observer error dynamics is obtained as 
\begin{align}
    \label{error-PDE} \tilde{u}_t(x,t) =& D \tilde{u}_{xx}(x,t) - a  \tilde{u}_x(x,t)  - g \tilde{u}(x,t)\nonumber \\
    &+p_1(x,l(t)) \tilde{u}_x(l(t),t), \\
\tilde{u}_x(0,t) = & 0,   \\
\tilde{u}(l(t),t)=& H^\top \tilde{X}(t), \\
\dot{\tilde{X}}(t)=& \left(A-LC\right)\tilde{X}(t).
\label{error-ODE} 
\end{align}

\begin{figure*}
 \centering  
\resizebox{.75\textwidth}{!}{
  \begin{tikzpicture}[scale=0.5]
  \centering
    \small
    % node placement with matrix library: 5x4 array
    {
      \node[block] at (-14,0) (F1) {$ \textbf{Controller} $};
      \node[block] at (-4,0) (f1)  {$\begin{matrix}\textbf{PDE} \\
            u_t (x,t) = D u_{xx} (x,t) - a u_x (x,t) - g u(x,t) , \\
u_x(0,t) =  U(t),  \\
u(l(t),t) =H^T X(t)  
          \end{matrix}$};
          \node[block] at (8,0) (f2)  {$\begin{matrix}\textbf{ODE} \\
\dot{X}(t) =  A X(t) + Bu_x (l(t), t) \end{matrix} 
          $};
          \node [] at (13,0) (u1) {}; 
      \node[block] at (1,-5) (o1) {$\begin{matrix}\textbf{Observer} \\
            \hat{u}_t (x,t) = D \hat{u}_{xx} (x,t) - a \hat{u}_x (x,t) - g \hat{u}(x,t)+p_1(x,l(t))\left(u_x(l(t),t)-\hat{u}_x(l(t),t)\right) , \\
\hat{u}_x(0,t) =  U(t),  \\
\hat{u}(l(t),t) =H^T \hat{X}(t), \\
\dot{\hat{X}}(t) =  A \hat{X}(t) + B u_x (l(t), t)+LC(X(t)-\hat{X}(t)) 
          \end{matrix}$};
      \node [] at (-13,-8) (b1) {};\&
      \node [] at (13,-8) (b2) {};\&
    };
    % now link the nodes
    \draw [connector] (f1.east) -- (f2);
    \draw [connector] (F1) -- node [above]{$U(t)$} (f1);
    \draw [line] (f2.south) |- ($(f1.south)-(0cm,0.5cm)$);
    \draw [connector] ($(f1.south)-(0cm,0.5cm)$) --  (f1.south);
    \draw [line] (f2) -| node[above right] {$\begin{matrix} \textbf{Measurements} \\ y_1(t)=u_x(l(t),t) \\
    y_2(t)=C X(t)\end{matrix}$}(u1);
    \draw [connector] ($(b2.east)+(1cm,3cm)$) --(o1.east) ;
    \draw [connector] ($(b1.west)-(1.3cm,0cm)$) -- node [left]{$l(t)$} ($(F1.south)-(0.5cm, 0cm)$); 
    \draw [connector] ($(o1.west)-(0cm,0.5cm)$) -| node[below,xshift=0.4cm]{$\hat{u}(x,t),\hat{X}(t)$}  ($(F1.south)+(0.8cm, 0cm)$);
    \draw [line] ($(b1.west)-(1.3cm,0cm)$) -- ($(b2.east)+(1cm,0cm)$);
    \draw [line,*-] ($(u1)+(1.25cm,0.2cm)$) -- ($(b2.east)+(1cm,0cm)$);
    \draw [connector] (u1.east) -- ($(u1.east)+(2cm, 0cm)$);
    \draw [connector,*->] ($(F1.east)+(0.5cm,0.2cm)$) |- ($(o1.west)+(0cm,0.5cm)$);
  \end{tikzpicture}
}
\caption{Block diagram of observer design and output-feedback system}
\label{blockdiag}
\end{figure*}
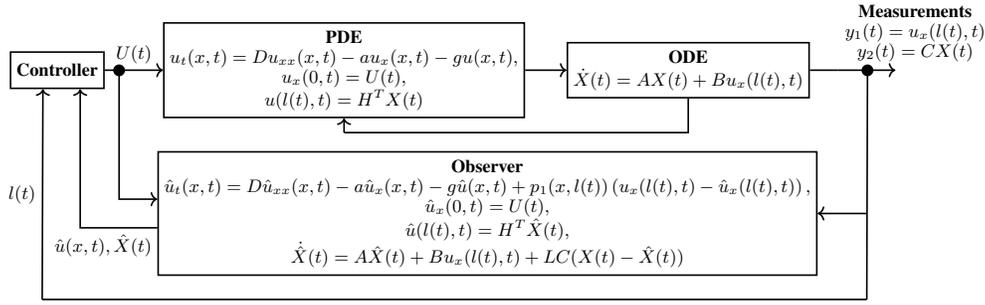

\subsubsection{Inverse backstepping transformation}
We consider the backstepping transformation in the inverse form as 
\begin{align}
    \tilde{u}(x,t)=\tilde{w}(x,t)+\int_x^{l(t)} P(x,y)\tilde{w}(y,t)dy,
    \label{error-bkst-inv}
\end{align}
where $P(x,y)\in \R$ is the gain kernel to be solved. Let the target system be 
\begin{align}
   \label{error-tar-PDE1} \tilde{w}_t(x,t)=&D\tilde{w}_{xx}(x,t)-a\tilde{w}_x(x,t)-(g+\lambda)\tilde{w}(x,t)\nonumber \\
   +\dot{l}(t)&\left(Q(x,l(t))-P(x,l(t))\right)\tilde{w}(l(t),t), \\
    \tilde{w}_x(0,t)=&\gamma_1\tilde{w}(0,t), \\
    \tilde{w}(l(t),t)=&H^\top \tilde{X}(t), \label{error-tar-PDE3} \\
    \dot{\tilde{X}}(t)=& \left(A-LC\right)\tilde{X}(t),
    \label{error-tar-ODE}
\end{align}
where $Q(x,y)\in \R$ is also the gain kernel obtained from direct backstepping transformation. Let the ODE observer gain $L$ be described as $L = \left[ 
 \begin{array}{cc}
 l_1 &
 l_2
 \end{array}  \right]^\top$. To make $A-LC$ Hurwitz, one can show the conditions for the gains as 
\begin{align}
    l_1>\frac{\tilde{a}l_2}{r_g}, \quad  l_2>\tilde{a}.
\end{align}
Taking the time and spatial derivatives of \eqref{error-bkst-inv} together with the solution of \eqref{error-tar-PDE1}-\eqref{error-tar-PDE3} and the stability of $\tilde{X}(t)$, we obtain \eqref{KernelPDE1}-\eqref{kernelPDE3}, and by choosing
    $P(x,y)=\tilde{P}(x,y)e^{\frac{a}{2D}(x+y)}$,
\eqref{KernelPDE1}-\eqref{kernelPDE3} become
\begin{align}
    \tilde{P}_{yy}(x,y)-\tilde{P}_{xx}(x,y)=&\frac{\lambda}{2D}\tilde{P}(x,y), \label{kernelPDE1-trans1}\\
    \tilde{P}(x,x)=&e^{-\frac{a x}{D}}\left(\frac{\lambda}{2D} x +\gamma_1\right),\\
    \tilde{P}_x(0,y)=&\frac{a}{2D}\tilde{P}(0,y),
    \label{kernelPDE3-trans1}
\end{align}
which cannot yield an analytical solution. To guarantee the well-posedness of the solution, we apply the method of successive approximation, by which a numerical solution can be obtained.

\subsection{Kernel PDE analysis by successive approximations}
The method of successive approximation is applied to prove that  \eqref{kernelPDE1-trans1}-\eqref{kernelPDE3-trans1} is well-posed, which leads to an integral equation for obtaining the solution. First, we apply the following change of spatial coordinate
$\bar x=y$, $\bar y=x$, and  $P^*(\bar x,\bar y)=\tilde{P}(x,y)
$. Then, we have
\begin{align} \label{kernelPDE1-trans2} 
    P_{\bar x \bar x}^*(\bar x,\bar y)-P_{\bar y \bar y}^*(\bar x,\bar y)=&\frac{\lambda}{2D}P^*(\bar x,\bar y),\\
    P^*(\bar x,\bar x)=& e^{-\frac{a\bar x}{D}}\left(\frac{\lambda}{2D}\bar x +\gamma_1\right),\\
    P_{\bar y}^*(\bar x,0)=&\frac{a}{2D}P^*(\bar x,0).\label{kernelPDE3-trans2} 
\end{align}
To convert gain kernel PDE to integral equation, we introduce
\begin{align} \label{coor-change} 
    \xi=\bar x+\bar y, \ \ \ \eta=\bar x-\bar y, \ \ \ P^*(\bar x, \bar y)=G(\xi,\eta),
\end{align}
where $(\xi,\eta) \in \mathcal{T}_1$ defined as $\mathcal{T}_1=\{\xi,\eta : 0<\xi<2l(t),\ 0<\eta <\text{min}(\xi,2l(t)-\xi)\}$. By \eqref{coor-change}, the conditions \eqref{kernelPDE1-trans2}--\eqref{kernelPDE1-trans2} are rewritten with respect to $G$ as 
\begin{align}
  G_{\xi \eta}(\xi,\eta)&=\frac{\lambda}{8D}G(\xi,\eta), \label{G-xi-eta}\\
  G(\xi,0)&=e^{-\frac{a}{2D}\xi}\left(\frac{\lambda}{4D}\xi+\gamma_1\right),\\
  G_{\xi}(\xi,\xi)-G_{\eta}(\xi,\xi)&=\frac{a}{2D}G(\xi,\xi). \label{G-xi-G-eta}
\end{align}
By applying the method of successive approximation, the solution to \eqref{G-xi-eta}-\eqref{G-xi-G-eta} is supposed to be in the form of 
$G(\xi,\eta)=G_0(\xi,\eta)+F[G](\xi,\eta)$,
where each component is
\begin{align}
    G_0&(\xi,\eta)=\frac{\lambda}{2D}e^{-\frac{a}{2D}\eta}\int_0^{\eta}f(\tau)d\tau+\frac{\lambda}{4D}\int_{\eta}^{\xi}e^{-\frac{a}{2D}\tau}f(\tau) d\tau  \label{G0-def} 
\\
    F[G]&(\xi,\eta)=\frac{\lambda}{4D}\int_0^{\eta}e^{\frac{a}{2D}(\tau-\eta)}\int_0^{\tau}G(\tau,s)ds d\tau\nonumber \\ &\quad \quad \quad    +\frac{\lambda}{8D}\int_{\eta}^{\xi}\int_0^{\eta}G(\tau,s)dsd\tau.
\end{align}
where $f(x)=\left(\left(1-\frac{a x}{2D}\right)-\frac{2a}{\lambda}\gamma_1\right)$. By using the definition of $G_0$, we let $G_{n+1}=F[G_n]$ and because of the positivity of $a, \ \gamma_1, \ D $, and $\lambda$, we get $\sup_{x\in [0,l[t]]} f(x)\leq 1$. 
Then, the bound of \eqref{G0-def} can be obtained as 
\begin{align}
    \left|G_0(\xi,\eta)\right|
    &\leq \frac{\lambda}{2}\left(\frac{1}{a}+ \frac{\bar l}{D} \right) =: M.
\end{align}
Then, one can obtain
\begin{align}
    \left|G_{n+1}(\xi,\eta)\right|
    &\leq M^{n+2}\frac{(\xi+\eta)^{n+1}}{(n+1)!}.
\end{align}
Thus, $G(\xi,\eta)=\sum_{n=0}^{\infty} G_n(\xi,\eta)$, which converges in $\mathcal{T}_1$ uniformly, and absolutely. In addition, it is continuous and twice differentiable. Thus, $G$ has a bound
\begin{align}
    |G(\xi,\eta)|\leq M e^{M(\xi+\eta)}.
\end{align}
By following the procedure described in \cite{smyshlyaev2004closed}, we get $G(\xi,\eta)$ is unique. Then, we can conclude the following result
\begin{align}
    |P(x,y)|\leq Me^{2Mx}, 
\end{align}
which guarantees the boundedness of the solution to the gain kernel PDE in \eqref{KernelPDE1}--\eqref{kernelPDE3}. 
\subsection{Direct backstepping transformation}
We use the following direct backstepping  transformation
\begin{align}
    \tilde{w}(x,t)=\tilde{u}(x,t)-\int_x^{l(t)} Q(x,y)\tilde{u}(y,t)dy.
    \label{error-bkst-fwd}
\end{align}
By applying \eqref{error-bkst-fwd} to the observer error system \eqref{error-PDE}--\eqref{error-ODE} and the target system \eqref{error-tar-PDE1}--\eqref{error-tar-ODE}, the conditions for the kernel function are obtained as 
\begin{align} \label{kernelPDE1-inv} 
    DQ_{xx}(x,y)-aQ_x(&x,y)-DQ_{yy}(x,y)\nonumber \\
    &-aQ_y(x,y)=\lambda Q(x,y), \\
    & \ \ \ \ \ \  Q(x,x)=-\frac{\lambda}{2D}x+\gamma_1, \\
    & \ \ \ \ \ Q_x(0,y)=0.\label{kernelPDE3-inv} 
\end{align}
To make sure the well-posedness of the solution to the kernel PDE above, we apply the following transformation $Q(x,y)=2e^{\frac{a}{2D}(x-y)}\tilde{Q}(x,y)$. Then, the conditions \eqref{kernelPDE1-inv}--\eqref{kernelPDE3-inv} can be rewritten as 
\begin{align}
    \tilde{Q}_{xx}(x,y)-\tilde{Q}_{yy}(x,y)=&\frac{\lambda}{D}\tilde{Q}(x,y), \\
    \tilde{Q}(x,x)=&-\frac{\lambda}{4D}x+\frac{\gamma_1}{2}, \\
    \tilde{Q}_x(0,y)=&\frac{a}{D}\tilde{Q}(0,y).
\end{align}
This kernel PDE is well-posed, so the solution of $Q(x,y)$ exists, which means direct transformation exists. As it is in the case of the inverse transformation, the closed-form solution of the direct kernel equation cannot be obtained. By applying the procedure in the previous section, we have the bound as $|Q(x,y)|\leq Me^{2Mx}$.

\subsection{Stability analysis of observer error system}

We consider the following Lyapunov function for the observer error target system 
\begin{align}
    \tilde{V}=\tilde{V}_{11}+\tilde{V}_{12}+d_2\tilde{V}_2+\frac{\gamma_1}{2} \tilde{w}(0,t)^2,
    \label{tilde-V-total}
\end{align}
where
\begin{align}\label{Vtilde11}
\tilde{V}_{11}=&\fr{1}{2}d_1 ||\tilde{w}||^2:=\fr{1}{2}d_1 \int_0^{l(t)} \tilde{w}(x,t)^2 dx, \\
\tilde{V}_{12}=&\fr{1}{2} ||\tilde{w}_x||^2:=\fr{1}{2} \int_0^{l(t)} \tilde{w}_x(x,t)^2 dx,  \\
\tilde{V}_2 =& \tilde{X}(t)^\top P \tilde{X}(t),\label{eqn:Vtil2}
\end{align}
and $P >0$ is a positive definite matrix satisfying the Lyapunov equation: 
\begin{align}
(A - LC )^\top P + P (A -LC )=-Q,
\end{align}
for some positive definite matrix $Q$. Since $(A-LC)$ is Hurwitz, positive definite matrices $P$ and $Q$ exist. In addition, we can denote that
$F(x,\tilde{X}(t))=\left(P(x,\tilde{X}(t)+l_{\rm s})-Q(x,\tilde{X}(t)+l_{\rm s})\right)H^\top\tilde{X}(t)$
Then, we state the following lemma. 
\begin{lemma}
Assume that \eqref{ineq-l} and \eqref{ineq-ldot} are satisfied for
\begin{align}
    \bar v=\min\left\{\frac{D}{8\bar l}, \frac{g+\lambda}{2\gamma_1}\right\},
\end{align}
for all time $t\geq 0$. Then, we conclude that for sufficiently large enough $d_1>0$ and $d_2>0$, there exists a positive constant
$\alpha_1=\min \left\{d_1\frac{D}{2},d_1\left(D+2\lambda\right),\left(g+2\lambda\right),\frac{\lambda_{\rm min}(Q)}{2\lambda_{\rm max}(P)} \right\}$
which satisfies the following norm estimates
\begin{align}
    \dot{\tilde{V}}\leq -\alpha_1\tilde{V}.
\end{align}
\label{lem:lamma1}
\end{lemma}
\begin{proof}
Taking the time derivative of the Lyapunov functions along the target system \eqref{error-tar-PDE1}--\eqref{error-tar-ODE}, we have
\begin{align}
\dot{\tilde{V}}_{11} =&
    d_1D\tilde{w}(l(t),t)\tilde{w}_x(l(t),t)-d_1\left(D\gamma_1-\frac{a}{2}\right)\tilde{w}(0,t)^2\nonumber \\ -d_1&\dot{l}(t)\int_0^{l(t)}\tilde{w}(x,t)F(x,\tilde{X}(t))dx-a\frac{d_1}{2}\tilde{w}(l(t),t)^2\nonumber \\ -d_1&D||\tilde{w}_x||^2-d_1(g+\lambda)||\tilde{w}||^2+d_1\frac{\dot{l}(t)}{2}\tilde{w}(l(t),t)^2  \label{eqn:V-tilde-dot-11}
\end{align}
\begin{align}
    \dot{\tilde{V}}_{12}=
    &H^\top(A-LC)\tilde{X}(t)\tilde{w}_x(l(t),t)-\frac{1}{2}\dot{l}(t)\tilde{w}_x(l(t),t)^2\nonumber \\ &-\tilde{w}_x(0,t)\tilde{w}_t(0,t)-D||\tilde{w}_{xx}||^2\nonumber \\ &+\int_0^{l(t)}a\tilde{w}_{xx}\tilde{w}_xdx-\left(g+\lambda\right)||\tilde{w}_x||^2\nonumber \\ &+(g+\lambda)\tilde{w}(l(t),t)\tilde{w}_x(l(t),t)-(g+\lambda)\gamma_1\tilde{w}(0,t)^2\nonumber \\
    &-\dot{l}(t)\int_0^{l(t)}\tilde{w}_{xx}(x,t)F(x,\tilde{X}(t))dx,    \label{eqn:V-tilde-dot-12}\\
\dot{\tilde{V}}_{2}=&-\tilde{X}(t)^\top Q \tilde{X}(t),  \label{eqn:V-tilde-dot-2}
\end{align}
where $d_1>0$. Applying Young's inequality and \eqref{error-tar-ODE}, and by using \eqref{error-tar-PDE3} to the time derivative of $\dot{\tilde{V}}_{11}$ in \eqref{eqn:V-tilde-dot-11} leads to
\begin{align}
    \dot{\tilde{V}}_{11} \leq&-d_1\left(D\gamma_1-\frac{a}{2}\right)\tilde{w}(0,t)^2-d_1D||\tilde{w}_x||^2\nonumber \\
    &-d_1(g+\lambda)||\tilde{w}||^2+\frac{D\epsilon_1}{2}\tilde{w}_x(l(t),t)^2\nonumber \\ &-d_1\dot{l}(t)\int_0^{l(t)}\tilde{w}(x,t)F(x,\tilde{X}(t))dx \nonumber \\
    &+\left(d_1^2\frac{D}{2\epsilon_1}+d_1\frac{\bar{l}(t)}{2}-d_1\frac{a}{2}\right)\tilde{w}(l(t),t)^2,
    \label{eqn:tar-err-V11}
\end{align}
where $\epsilon_1> 0$ is an arbitrarily small constant. Similarly, using Agmon's inequalities and Young's inequality (defined in \cite{demir2021neuron}) into the time derivative of $\tilde{V}_{12}$ in \eqref{eqn:V-tilde-dot-12} gives us
\begin{align}
    \dot{\tilde{V}}_{12}\leq & \gamma_1^2\bigg(D\epsilon_1+\epsilon_2+\epsilon_3(g+\lambda)+\bar v+\frac{\bar v \epsilon_5}{2}\bigg)\tilde{w}(0,t)^2\nonumber \\
    &-\bigg(\gamma_1(g+\lambda)\bigg)\tilde{w}(0,t)^2-\bigg(D-\frac{D}{4}\bigg)||\tilde{w}_{xx}||^2\nonumber \\
    &-\bigg(-2\bar l(D\epsilon_1+\epsilon_2+\epsilon_3(g+\lambda)+\epsilon_4\bar v+\bar v)\bigg)||\tilde{w}_{xx}||^2\nonumber \\
    &-\left((g+\lambda)-\frac{a^2}{D}\right)||\tilde{w}_x||^2-\gamma_1\tilde{w}(0,t)\tilde{w}_t(0,t)\nonumber \\
    &+\bigg(\left(d_1^2\frac{D}{2\epsilon_1}+d_1\frac{a}{2}+d_1\frac{\bar v}{2}+\frac{(g+\lambda)}{2\epsilon_3}\right)\lambda_{\rm max}(HH^\top)\nonumber \\
    &+\frac{1}{2\epsilon_2}\lambda_{\rm max}((H^\top(A-LC))^2)\bigg)\tilde{X}^\top \tilde{X}\nonumber \\
    &+\frac{|\dot{l}(t)|}{2\epsilon_4}F(l(t),\tilde{X}(t))^2+\frac{|\dot{l}(t)|}{2\epsilon_5}\gamma_1 F(0,\tilde{X}(t))^2\nonumber \\
    &+|\dot{l}(t)|\int_0^{l(t)}\tilde{w}_{x}(x,t)F_x(x,\tilde{X}(t))dx,
    \label{eqn:tar-err-V12}
\end{align}
where $\epsilon_i>0$ for $i=\{1,...,5\}$ are arbitrarily small constants. The time derivative of $\tilde{V}_2$ in \eqref{eqn:V-tilde-dot-2} is given by
\begin{align}
     \dot{\tilde{V}}_2\leq -\lambda_{\rm min}(Q)\tilde{X}^\top\tilde{X}.
     \label{eqn:tar-err-V2}
\end{align}
Then, we use Young's and Cauchy-Schwarz Inequalities  for $F(.,.)$ terms. There exists positive constants $L_i>0$, for $i=1,2,3,4$, we have $F(0,\tilde{X}(t))^2\leq L_1 \tilde{X}^\top\tilde{X}$, $F(l(t),\tilde{X}(t))^2\leq L_2 \tilde{X}^\top\tilde{X}$, $\int_0^{l(t)}F(x,\tilde{X}(t))^2dx\leq L_3 \tilde{X}^\top\tilde{X}$, and $\int_0^{l(t)}F_x(x,\tilde{X}(t))^2dx\leq L_4 \tilde{X}^\top\tilde{X}$. With these inequalities and \eqref{eqn:tar-err-V11}-\eqref{eqn:tar-err-V2}, \eqref{eqn:tar-err-V12} becomes
\begin{align}
    \dot{\tilde{V}}\leq& -\frac{D}{2}||\tilde{w}_{xx}||^2-d_1\frac{D\gamma_1}{2}\tilde{w}(0,t)^2-d_1(g+\lambda)||\tilde{w}||^2\nonumber \\
    &+\frac{\bar v\epsilon_6}{2}||\tilde{w}||^2-\left(d_1D+(g+\lambda)-\frac{a^2}{D}-\frac{\bar v \epsilon_7}{2}\right)||\tilde{w}_x||^2\nonumber \\
    &-\gamma_1\tilde{w}(0,t)\tilde{w}_t(0,t)+d_1\frac{\bar v}{2}L_1|\tilde{X}(t)|^2+\frac{\bar v}{2\epsilon_4}L_2\bar v^2\nonumber \\
    &+d_1^2\frac{\bar v}{2\epsilon_6}L_3\bar v^2+\frac{\bar v}{2\epsilon_7}L_4|\tilde{X}(t)|^2-d_2\lambda_{\rm min}(Q)|\tilde{X}(t)|^2\nonumber \\
    &+\bigg(\left(d_1^2\frac{D}{2\epsilon_1}+d_1\frac{a}{2}+d_1\frac{\bar v}{2}+\frac{(g+\lambda)}{2\epsilon_3}\right)\lambda_{\rm max}(HH^\top)\nonumber \\
    &+\frac{1}{2\epsilon_2}\lambda_{\rm max}((H^\top(A-LC))^2)\bigg)|\tilde{X}(t)|^2.
    \label{eqn:tot-tilde-V-ineq}
\end{align}
By using positive definiteness, we recall
\begin{align}
    \label{ineq-XPX} 
    \lam_{\rm min}(P) X^\top X \leq X^\top P X \leq \lam_{\rm max}(P) X^\top X,
\end{align}
where $\lam_{\rm min}(P) >0$ and $\lam_{\rm max}(P)>0$ are the smallest and the largest eigenvalues of $P$. Finally, by recalling $\gamma_1 \geq \frac{D}{a}$, we can choose constants $d_1$ and $d_2$ as
\begin{align}
  \label{eqn:d1-con}
    d_1 \geq &\frac{2a^2+D\bar v \epsilon_7}{D^2},
\\
    d_2 \geq &\frac{2}{\lambda_{\rm min}(Q)}\bigg(\left(\frac{D}{2\epsilon_1}+d_1\frac{a+\bar v}{2}\right)\lambda_{\rm max}(HH^\top)\nonumber \\
    &+\frac{(g+\lambda)}{2\epsilon_3}\lambda_{\rm max}(HH^\top)+\frac{\lambda_{\rm max}((H^\top(A-LC))^2)}{2\epsilon_2}\nonumber \\
    &+\left(\frac{d_1}{2}L_1+\frac{1}{2\epsilon_4}L_2+\frac{d_1^2}{2\epsilon_6}L_3+\frac{1}{2\epsilon_7}L_4\right)\bar v\bigg).
\end{align}
Thus, one can show that \eqref{eqn:tot-tilde-V-ineq} leads to
\begin{align}
    \dot{\tilde{V}}\leq &-d_1\frac{D\gamma_1}{2}\tilde{w}(0,t)^2-d_1\left(D+2\lambda\right)\tilde{V}_{12}-\left(g+2\lambda\right)\tilde{V}_{11}\nonumber \\
    &-d_2\frac{\lambda_{\rm min}(Q)}{2\lambda_{\rm max}(P)}\tilde{V}_2 \nonumber \\
    \leq& -\alpha_1 \tilde{V}.
    \label{eqn:Vtilde-final}
\end{align}
Thus, Lemma 1 holds.
\end{proof} Hence, the target $\tilde{w}$-system \eqref{error-tar-PDE1}-\eqref{error-tar-ODE} is exponentially stable at the origin. Because of the invertibility of backstepping transformation, the stability of $\tilde{w}$-system promotes to the original $\tilde{u}$-system \eqref{error-PDE}-\eqref{error-ODE} exponentially stable. This completes the proof of Theorem \ref{thm:teorem1}.

\section{OBSERVER-BASED OUTPUT-FEEDBACK CONTROL}
\label{sec:section4}
In this section, an output-feedback control law is constructed using the observer designed in Section \ref{sec:observer} with the measurements \eqref{measurements}, and the following theorem holds.
\begin{theorem}
\label{thm:teorem2}
Consider the closed-loop system  \eqref{ulin-PDE}-\eqref{ulin-ODE} with the measurements \eqref{measurements}, and the observer \eqref{obulin-PDE}-\eqref{obulin-ODE} under the output-feedback control law:
\begin{align}
    &U(t)=\frac{D\gamma_2-\beta}{D}\hat{u}(0,t)+\phi'(-l(t))^\top-\gamma_2\phi(-l(t))^\top\hat{X}(t)\nonumber \\ &  \quad \ \ \ \  -\frac{1}{D}\int_0^{l(t)}(\phi'(-y)^\top-\gamma_2\phi(-y)^\top) B\hat{u}(y,t)dy
    \label{eqn:controller}
\end{align}
where $\gamma_2\geq \frac{a}{D}$ and $\phi(x)$ is
\begin{align}
    \phi(x)^\top=\begin{bmatrix}H^\top & K^\top-\frac{1}{D}H^\top BH^\top\end{bmatrix}e^{N_1x}\begin{bmatrix} I \\ 0
\end{bmatrix}.
\end{align}
The matrices $K = [ k_1 \quad k_2]$ and $N_1 \in \R^{4 \times 4}$ is defined as 
\begin{align}
    N_1=\begin{bmatrix}0 & \frac{1}{D}\left(gI+A+\frac{a}{D}BH^\top\right)\\ I &\frac{1}{D}\left(BH^\top+aI\right)\end{bmatrix},
\end{align}
where $k_1 > \frac{\tilde a}{\beta}$, and $k_2 > 0$. Then, there exist positive constants $\bar{M}>0$, $\kappa>0$ and $\zeta>0$ such that if $\Phi(0)<\bar{M}$ where
$\Phi(t):=||u||_{\text{H}_1}^2+|X|^2+||\hat{u}||_{\text{H}_1}^2+|\hat{X}|^2$,
then the following norm estimate holds:
\begin{align}
    \Phi(t)\leq \zeta\Phi(0)\exp\left(-\kappa t\right). 
\end{align}
Namely, the closed-loop system is locally stable in the sense of $\mathcal{H}_1$-norm.
\end{theorem}
\subsection{Backstepping transformation}
The following transformation from $(\hat{u},\hat{X})$ into $(\hat w, \hat X)$ is implemented by using the gain kernels proposed in \cite{demir2021neuron}
\begin{align} 
\label{bkst}
\hat{w}(x,t) = &\hat{u}(x,t) - \int_x^{l(t)} k(x,y) \hat{u}(y,t) dy \nonumber \\
&- \phi(x - l(t))^T \hat{X}(t), 
\end{align}
and the reverse transformation is
\begin{align}
    \hat{u}(x,t)=&\hat{w}(x,t)+\int_{x}^{l(t)}q(x,y)\hat{w}(y,t)dy\nonumber \\&+\varphi(x-l(t))^\top \hat{X}(t).
    \label{bkst-back}
\end{align}
Taking the time and spatial derivatives of the transformation above, the target $\hat w$-system is obtained as 

\begin{align}
    \hat{w}_t=&D\hat{w}_{xx}-a\hat{w}_x-g\hat{w} +\dot l(t) E(x,\hat{X}(t))  \nonumber \\
    &+p_1(x,l(t))\tilde{u}_x(l(t),t)\nonumber \\
    &-\int_x^{l(t)}k(x,y)p_1(y,l(t))\tilde{u}_x(l(t),t)dy, \label{eqn:targetwhat-t} \\
    \hat{w}_x(0,t)=&\gamma_2 \hat{w}(0,t), \label{eqn:whbound1}\\
    \hat{w}(l(t),t)=&0, \label{eqn:whbound2}\\
    \dot{\hat{X}}(t)=&(A+BK)\hat{X}(t)+B\hat{w}_x(l(t),t)\nonumber \\
    &+B\tilde{u}_x(l(t),t)+LC\tilde{X}(t),
\end{align}
where we denote $E(x,\hat{X}(t)) =(\phi'(x-\hat{X}(t) - l_{\rm s})^\top-k(x,\hat{X}(t) + l_{\rm s}) H^\top)\hat{X}(t)$ . By evaluating the spatial derivative of \eqref{bkst} at $x=0$, we derive the control law as in \eqref{eqn:controller}.

The gain kernels $k(x,y)$ and $\phi(x)$ are derived in \cite{demir2021neuron}.
\subsection{Stability analysis}

Define the Lypunov function for closed-loop system as
\begin{align}
    V_{\rm tot}=c_1\tilde{V}(t)+\hat{V}_{11}+\hat{V}_{12}+\frac{1}{2}\gamma_2 \hat{w}(0,t)^2+d_4\hat{V}_2
    \label{Vtot1}
\end{align}
where $c_1>0$ is chosen to sufficiently large, $\tilde{V}(t)$ is defined in \eqref{tilde-V-total}--\eqref{eqn:Vtil2} and each term in $\hat{V}(t)$ is
\begin{small}
\begin{align}
    \hat{V}_{11}(t)=&\fr{1}{2}d_3 ||\hat{w}||^2:=\fr{1}{2}d_3 \int_0^{l(t)} \hat{w}(x,t)^2 dx, \\
    \hat{V}_{12}(t)=&\fr{1}{2} ||\hat{w}_x||^2:=\fr{1}{2} \int_0^{l(t)} \hat{w}_x(x,t)^2 dx, \\   \hat{V}_{2}(t)=&\hat{X}(t)^\top  \hat{P} \hat{X}(t). 
\end{align}
\end{small}
Lyapunov function of closed-loop system is also written as
\begin{align}
    V_{\rm tot}(t)=& c_1\frac{1}{2}\left(d_1||\tilde{w}||^2+||\tilde{w}_x||^2\right)+\frac{1}{2}\left(d_3 ||\hat{w}||^2+||\hat{w}_x||^2\right)\nonumber \\
    &+c_1d_2\tilde{X}(t)^\top P \tilde{X}(t)+d_4\hat{X}(t)^\top  \hat{P} \hat{X}(t)\nonumber \\
    &+\frac{1}{2}\left(c_1\gamma_1\tilde{w}(0,t)^2+\gamma_2 \hat{w}(0,t)^2\right).
\label{Vtot}
\end{align}Then, we state the following lemma.
\begin{lemma}
\label{lem:lemma2}
Let Assumptions \ref{asm:assump1} and \ref{asm:assump2} hold with 
\begin{align}
    \bar{v}\leq\min\left\{ \frac{g}{3\gamma_2},\frac{D}{8\bar l}, \frac{g+\lambda}{2\gamma_1}\right\},
\end{align}
for all time $t\geq 0$. Then, for sufficiently large enough $d_3>0$ and $d_4>0$, there exists a positive constant $\alpha>0$ and $\beta>0$ such that the following norm estimates holds
\begin{align}
    \dot{V}_{\rm tot}   \leq& -\alpha V_{\rm tot}+\beta V_{\rm tot}^{3/2}. 
\end{align}
%which guarantees the local stability of $(\hat{X},\hat{w},\tilde{X},\tilde{w})$-system.
\end{lemma}
\begin{proof}
By applying Young's, Cauchy-Schwarz, Poincare's, and Agmon's inequalities, with the help of Assumption \ref{asm:assump1} and \ref{asm:assump2}, and following the same strategy that is offered in the \cite{demir2021neuron} for the Lyapunov analysis, one can derive
\begin{align}
    \dot{V}_{\rm tot}   \leq& -\alpha V_{\rm tot}+\beta V_{\rm tot}^{3/2},
    \label{eqn:result}
\end{align}
for 
\begin{align}
    \alpha&=\min\left\{\frac{1}{2}\alpha_1,d_3\frac{D}{2},\frac{g}{2},d_4\fr{\lambda_{\rm min}(\hat{Q})}{4\lambda_{\rm max}(\hat{P})}\right\}, \\
    \beta&=\frac{r_{\rm g} e_1^\top}{\lambda_{\rm min}(\hat{P})} \left(\frac{1}{2} L_5+\frac{1}{2}L_6+d_3\frac{1}{2\varepsilon_3}L_7+\frac{1}{2}L_8\right),
    \label{eqn:beta-hat}
\end{align}
where $\varepsilon_3\leq \frac{g}{\bar v}$ such that Lemma \ref{lem:lemma2} holds.
\end{proof}
To prove local stability, Lemma 2 from \cite{demir2021neuron} guarantees to the convergence in all time. It satisfies that $V_{\rm tot}(t)<M_1$ holds for some $M>0$, then $|\tilde{X}|<r$. Lemma 3 from \cite{demir2021neuron} proves that for \eqref{eqn:result}, if $V_{\rm tot}(0)<M$, then $V_{\rm tot}(t)<M$ for all $t>0$. In addition, $\dot{l}(t)$ can be written as $\dot{l}(t)=r_{\rm g}e_1^\top X(t)$, so we can bound $\dot{l}(t)$ to handle in the norm equivalence as
\begin{align}
    |\dot{l}(t)|&\leq r_{\rm g}e_1^\top \left(\sqrt{\frac{\tilde{V}_2}{\lambda_{\rm min }(P)}}+\sqrt{\frac{\hat{V}_2}{\lambda_{\rm min}(\hat{P})}}\right)
    \end{align}
This leads us $|\dot{l}(t)|^2\leq \delta^2 V_{\rm tot}(t)$.
Thus,
\begin{align}
    V_{\rm tot}(t)\leq V_{\rm tot}(0)\exp\left(-\frac{\alpha}{2}t\right).
\end{align}
The norm equivalence between the target and original systems is shown using the direct and inverse transformations of both observer target and observer error target systems. Taking square of  \eqref{error-bkst-inv}, \eqref{error-bkst-fwd}, \eqref{bkst}, and \eqref{bkst-back} and applying Young's and Cauchy-Schwarz inequalities, and integrating on $[0,l(t)]$, one can get the norm inequalities for the target system $(\hat{w},\hat{w}_x,\tilde{w},\tilde{w}_x)$. 
Then, let
$\Psi= ||\hat{w}||_{\text{H}_1}^2+|\hat{X}|^2+||\tilde{w}||_{\text{H}_1}^2+|\tilde{X}|^2$.
Using the norm inequalities for each term in $\Psi$ and \eqref{Vtot}, one can obtain $\overline{M}>0$ and $\underline{M}<0$ such that
\begin{align}
    \underline{M}\Psi(t)\leq V_{\rm tot}(t)\leq \overline{M}\Psi(t)
\end{align}
holds. Therefore, we get
\begin{align}
    \Psi(t)\leq \frac{\overline{M}}{\underline{M}}\exp\left(-\frac{\alpha}{2}t\right)\Psi(0)
\end{align}
Now, we apply the norm equivalence argument to the transformations between the target system and the original system by using $u(x,t)=\tilde{u}(x,t)+\hat{u}(x,t)$ and $X(t)=\tilde{X}(t)+\hat{X}(t)$. Let
$\Phi(t)=||u||_{\text{H}_1}^2+|X|^2+||\hat{u}||_{\text{H}_1}^2+|\hat{X}|^2$
so, one can obtain $\overline{N}>0$, and $\underline{N}<0$ such that
\begin{align}
    \underline{N}\Phi(t)\leq \Psi(t)\leq \overline{N}\Phi(t)
\end{align}

Therefore, we get
\begin{align}
    \Phi(t)\leq \frac{\overline{N}}{\underline{N}}\exp\left(-\frac{\alpha}{2}t\right)\Phi(0)
\end{align}
Namely, since the backstepping transformation for the target error system and for the observer system are invertible, the local stability of $(\tilde{w},\tilde{X},\hat{w},\hat{X})$ guarantees the local stability of $(u,X,\hat{u},\hat{X})$, which completes the proof of Theorem \ref{thm:teorem2}.

\section{NUMERICAL SIMULATION}

\label{sec:section5}
Numerical simulation is performed for the plant \eqref{ulin-PDE}-\eqref{ulin-ODE}, and observer \eqref{obulin-PDE}-\eqref{obulin-ODE} with a designed control law \eqref{eqn:controller}. We use the biological constants proposed in \cite{diehl2014one}, as shown in Table \ref{tab:initial}. The initial conditions for the plant is set as $c_0(x) =2c_{\infty}$ for tubulin concentration along the axon, and $l_0=1 \mu m$ for the axon length. In addition, initial conditions for the observer are chosen $c_{\rm o}(x,0)=0$ for all along the axon where the subscript $\rm o$ denotes the observer. The control and observer gains for ODE parts of the closed-loop system are
$k_1=-0.1,\quad k_2=10^{13},\quad l_1=1$, and $l_2=0.1$. The gain parameter, $\lambda$, is chosen $0.05$ to obtain fast convergence.

\begin{figure}[!t]
    \centering
    \subfloat[{The axon length converges to the desired length. }\label{subfig-1}]{%
\includegraphics[width=0.95\linewidth]{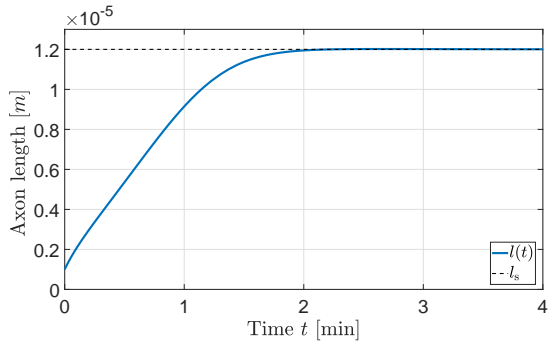}}
    \vfill
\subfloat[{The estimated tubulin concentration, $c_{\rm o}(x,t)=\hat{u}(x,t)+c_{\rm eq}(x)$, generated by the observer in \eqref{obulin-PDE}--\eqref{obulin-ODE} and shown in 
Fig. \ref{blockdiag}, converges all along the axon to the actual, unmeasured concentration $c(x,t)$ by about $t=0.75$ min. Thereafter, both $c_{\rm o}$ and $c$ converge together to the steady-state $c_{\rm eq}(x)$ from $t=0.75$ min till about $t=2$ min. Note that $c_{\rm o}$ converges to $c$ faster than $c$ converges to $c_{\rm eq}$ and faster than $l$ converges to $l_{\rm s}$}\label{subfig-2}]{ \includegraphics[width=0.95\linewidth]{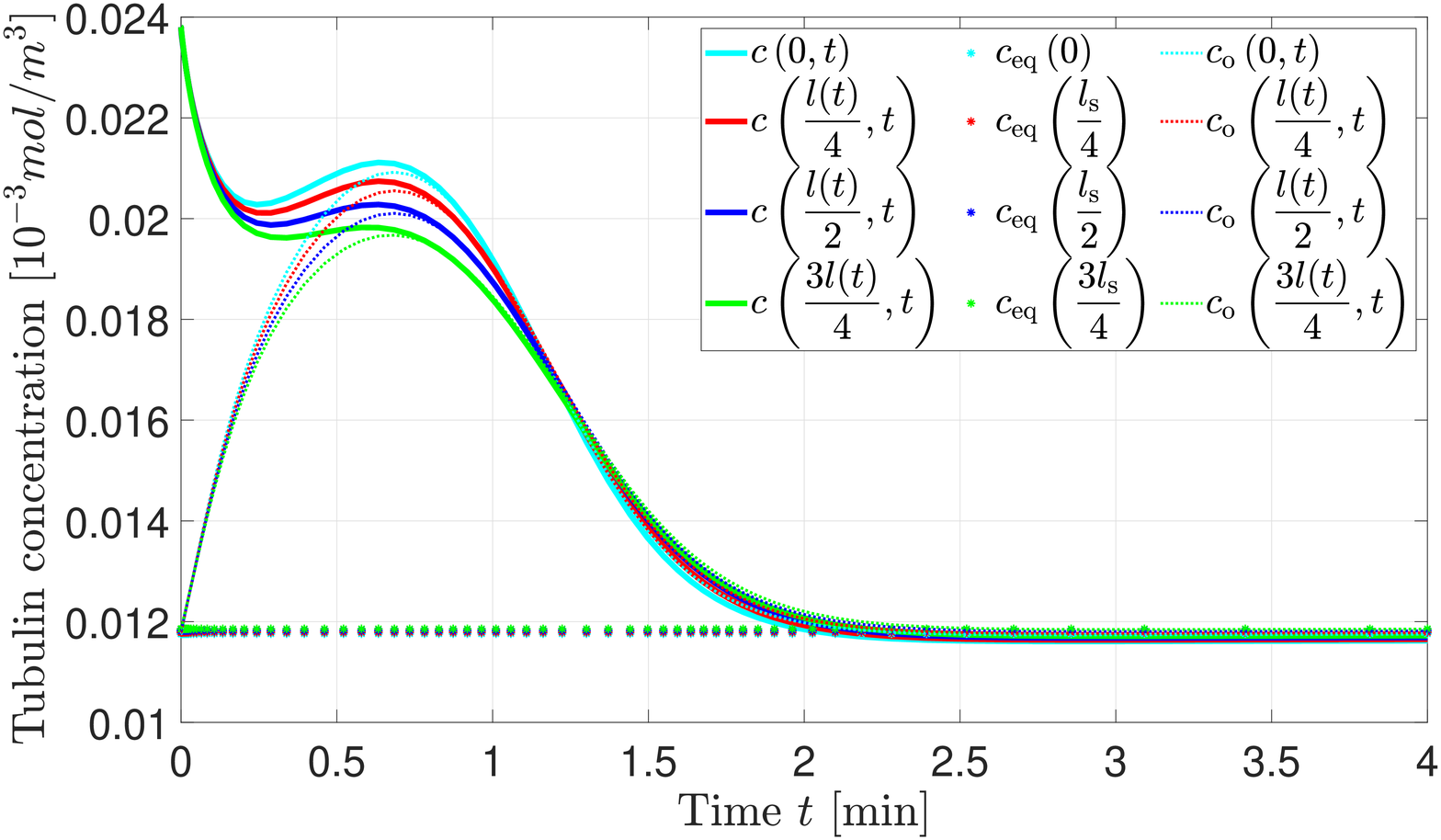}
}
    \caption{Closed-loop  output-feedback response. }
\end{figure}
Fig. \ref{subfig-1} shows that the axon length converges to the desired axon length $l_{\rm s}$. Fig. \ref{subfig-2} illustrates that the observer states converge to the actual tubulin concentration, by which the estimation of the tubulin concentration along the axon is successfully achieved. In addition, the tubulin concentration along the axon converges to the steady-state solution, which demonstrates the effectiveness of the proposed output-feedback control law.

\begin{table}[t]
\hfill
	\caption{\label{tab:initial}Biological constants and control parameters}
	\centering
	\begin{tabular}{cccc} 
		\hline
		Parameter  & Value  & Parameter & Value \\
		\hline
		$D$ & $10\times10^{-6}  m^2/s$& $\tilde{r}_{\rm g}$ & $0.053$\\
		$a$& $1\times 10^{-8}  m/s$ & $\gamma$ &  $10^4$\\
		$g$& $5\times 10^{-7} \ s^{-1}$ & $l_{\rm c}$ & $4\mu m$\\
		$r_{\rm g}$& $1.783\times 10^{-5} \ m^4/(mol s)$ & $l_s$ & $12\mu m$ \\
		$c_{\infty}$ &  $0.0119  \ mol/m^3$ & $l_0$& $1\mu m$ \\
 \hline
	\end{tabular}
\end{table} 

\section{CONCLUSIONS}
This study proposes a novel output-feedback control for a coupled PDE-ODE dynamics with a moving boundary for neuron growth model by applying the PDE backstepping technique. We proposed a PDE-observer with designing an observer gain via backstepping, with showing the stability analysis to estimate the unknown states in the plant dynamics. Then, we designed the output-feedback controller for the axon elongation problem to achieve the desired length of the axon. Next, the stability analysis of the closed-loop system under the proposed output-feedback controller is provided. Finally, we have verified our theoretical results in the simulation results using the biological parameters.

For further research, it is also feasible to establish the local stability without applying linearization with the output-feedback control law. Further studies about the relation between ECM and tubulin production would make it the controller applicable in clinical trials such as ChABC.
\label{sec:section6}

\bibliographystyle{IEEEtranS}
\bibliography{BIB_CDC21.bib}
\end{document}